\documentclass[12pt]{article}
\usepackage{amssymb,amsmath,amsthm,mathrsfs}
\usepackage[usenames]{color}
\usepackage{hyperref}
\usepackage[all]{xy}
\usepackage{xypic}
\usepackage{graphicx}
\usepackage{amscd}
\usepackage{enumerate}
\usepackage{enumitem}

\newtheorem{theorem}{Theorem}[section]

\newtheorem{corollary}[theorem]{Corollary}
\newtheorem{lemma}[theorem]{Lemma}

\newtheorem*{lemma*}{Lemma}
\newtheorem{definition}[theorem]{Definition}

\newtheorem{conjecture}[theorem]{Conjecture}

\newtheorem{claim}[theorem]{Claim}

\def\C{\Bbb C}
\def\N{\Bbb N}

\def\Z{\Bbb Z}

\DeclareMathOperator{\Cent}{Cent}

\DeclareMathOperator{\Mat}{Mat}

\DeclareMathOperator{\SL}{SL}

\DeclareMathOperator{\Span}{Span}

\DeclareMathOperator{\gcl}{gcl}

\DeclareMathOperator{\PSL}{PSL}
\DeclareMathOperator{\trace}{trace}
\DeclareMathOperator{\Ad}{Ad}
\DeclareMathOperator{\spec}{spec}
\DeclareMathOperator{\pr}{pr}

\title{Is being a higher rank lattice a first order property?}
\author{Nir Avni and Chen Meiri}
\date{\today}

\begin{document}

\maketitle

\begin{abstract} We show that there is a sentence $\varphi$ in the first order language of groups such that a finitely generated group $\Gamma$ satisfies $\varphi$ if and only if $\Gamma$ is isomorphic to a group of the form $\PSL_n(O)$, where $n \geq 3$ and $O$ is a ring of $S$-integers in a number field.
\end{abstract}

\section{Introduction}
This paper is a continuation of our work \cite{AM} on the model theory of higher rank lattices. Whereas \cite{AM} focuses on  individual higher rank lattices, our interest here is in the collection of all higher rank lattices. One motivation for this work is the following conjecture:

\begin{conjecture} \label{conj:fop} The property of being a higher rank arithmetic group is a first order property. That is, there is a first order sentence $\varphi$ in the language of groups such that, for any finitely generated group $\Delta$, the sentence $\varphi$ holds in $\Delta$ if and only if $\Delta$ is a higher rank arithmetic group.
\end{conjecture} 

In a forthcoming paper, we prove the conjecture with respect to a large class of higher rank arithmetic groups (but not all of them). In this note, we prove it for the collection of all groups of the form $\SL_n(O)$, where $n \geq 3$ and $O$ is the ring of $S$-integers in a number field $K$.

Our approach to Conjecture \ref{conj:fop} is via bi-interpretation with the integers. The main result of \cite{AM} is that, if $\Gamma$ is a higher rank lattice which is either non-uniform or uniform of orthogonal type and dimension at least 9, then $\Gamma$ is bi-interpretable with the ring $\mathbb{Z}$ of integers. The main result of this paper is:

\begin{theorem} \label{thm:main.bi-int} Let $(\Gamma_n)_n$ be an enumeration of all groups of the form $\PSL_d(O)$, where $d \geq 3$ and $O$ in the ring of $S$-integers in a number field $K$, where $S$ is a finite set of places of $K$ containing all the archimedean ones. Then the sequence $(\Gamma_n)_n$ is bi-interpretable with $\mathbb{Z}$.
\end{theorem} 

For the notion of bi-interpretability of sequences, see Definition \ref{def:logic}.\\

The same arguments as in \cite[7.7]{Nie} (who attributes the result to \cite[Lemma 1]{Kh}) give

\begin{corollary} \label{cor:fop} There is a first order sentence $\varphi$ in the language of groups such that, for any finitely generated group $\Delta$, the sentence $\varphi$ holds in $\Delta$ if and only if $\Delta$ is isomorphic to a group of the form $\PSL_d(O)$, where $d \geq 3$ and $O$ in the ring of $S$-integers in a number field.
\end{corollary} 



Another corollary of Theorem \ref{thm:main.bi-int} is

\begin{corollary}\label{cor:main2} Let  $(\Gamma_n)_n$ be as in Theorem  \ref{thm:main.bi-int}. There is a single formula $\phi(\underline{x},y)$ such that, for every $n \in \N$ and every finitely generated subgroup $\Delta \subseteq \Gamma_n$, there is $\underline{\alpha} \in \Gamma_n \times \cdots \times \Gamma_n$ such that $$\Delta=\{\beta \mid \Gamma_n \text{ satisfies } \psi(\underline{\alpha},\beta)\}.$$
\end{corollary} 


\section{Model theory preliminaries}

\begin{definition} \label{def:logic} Let $L$ be a first order language, and let $(M_n)_n$ be a sequence of $L$-structures. \begin{enumerate} 
\item Let $k\in \mathbb{N}$, and, for every $n$, let $A_n$ be a subset of $M_n^k$. We say that the sequence $(A_n)_n$ is a {\em definable sequence of subsets} of $(M_n^k)_n$ if there are $\ell \in \mathbb{N}$, a first order formula $\varphi(x_1,\ldots,x_k,y_1,\ldots,y_\ell)$, and a sequence $c_n \in A_n^\ell$ such that $A_n=\left\{ a\in M_n^k \mid \varphi(a,c_n)\text{ holds in $M_n$} \right\}$, for every $n$. The notions of a {\em definable sequence of functions} and a {\em definable sequence of imaginaries} are analogous.
\item Let $k\in \mathbb{N}$, and, for every $n$, let $\mathfrak{A}_n$ be a collection of subsets of $M_n^k$. We say that the sequence $(\mathfrak{A}_n)_n$ is a {\em uniformly definable collection} if there are natural numbers $k,\ell$, a definable sequence $(Y_n)_n$ of subsets of $(M_n^\ell)_n$, and a definable sequence $(X_n)_n$ of $(M_n^{k} \times Y_n)_n$, such that, for every $n$, 
\[
\mathfrak{A}_n = \left\{ \left\{ a\in M_n^k \mid (a,c)\in X_n \right\} \mid c\in Y_n \right\}
\]
\item Let $L'$ be a (possibly different) first order language, and let $(M_n')_n$ be a sequence of $L'$-structures. An interpretation of $(M_n')_n$ in $(M_n)_n$ consists of the following data: \begin{enumerate} 
\item A definable sequence $(H_n)_n$ of imaginaries in $(M_n)_n$.
\item For every constant symbol $c$ in $L'$, a sequence $(c_n)_n$ such that $c_n \in H_n$.
\item For every relation symbol $R$ in $L'$, a definable relation $(R_n)_n$ on $(H_n)_n$ of the same arity.
\item For every function symbol $f$ in $L'$, a definable function $(f_n)_n$ on $(H_n)_n$ or the same arity.
\item A sequence of functions $h_n:H_n \rightarrow M_n'$.
\end{enumerate} 
such that, for every $n$, the function $h_n$ induces an isomorphism between the $L'$-structures $M_n'$ and $H_n$.
\item A self interpretation $\mathcal{H}=(H_n,\ldots,h_n)_n$ of $(M_n)_n$ is called trivial if the maps $h_n:H_n \rightarrow M_n$ are definable.
\end{enumerate} 
\end{definition} 


We will use the following known theorem:

\begin{theorem} \label{thm:prelim.logic} Let $(O_n)$ be a sequence of rings of $S$-integers in number fields. \begin{enumerate} 
\item The sequence $(\mathbb{Z})_n$ of subsets of $(O_n)_n$ is uniformly definable.
\item Let $O_n$ be an enumeration of all rings of $S$-integers in number fields. Then $(O_n)_n$ is bi-interpretable with $(\mathbb{Z})_n$. Moreover, there is an interpretation $((B_n,\oplus_n,\otimes_n,b_n))_n$ of $(O_n)_n$ in $(\mathbb{Z})_n$ such that  for every $n$, $B_n=\Z$, $b_n: \mathbb{Z} \rightarrow O_n$ is a bijection,  $b_n(0)=0$,  $b_n(1)=1$ and  $(b_n)_n$ is a definable sequence of functions in $(O_n)_n$, where we identify $\mathbb{Z}$ as a subring of $O_n$.

\item Any self interpretation of the sequence $(\mathbb{Z})_n$ is trivial.
\end{enumerate} 
\end{theorem}

\section{Notation}

\begin{enumerate}
\item For a ring $A$, let $\PSL_n(A):=\SL_n(A)/Z(\SL_n(A))$. For every $1 \le i \ne j \le n$ and every $a \in A$,  $e^n_{i,j}(a) \in \PSL_n(A)$ is the image of the  matrix with $1$ on the diagonal, $a$ in the $(i,j)$-entry and zero elsewhere and  $E^n_{i,j}(A):=\{e^n_{i,j}(a)\mid a \in A\}$.
\item If $O$ is the ring of $S$-integers in some number field $K$ and $\alpha \in \PSL_n(O)$, $\trace(\alpha)$ is either zero or an element of $K^\times/U_n$ where $U_n$ is the subgroup of $n$-th roots of unity of $K^\times$.  
\item Let $(d_n,O_n)$ be an enumeration of all pairs $(d,O)$, where $d\in \mathbb{Z}_{\geq 3}$ and $O_n$ is the ring of $S$-integers in some number field. Denote $\Gamma_n:=\PSL_{d_n}(O_n)$, $\epsilon_n(a):=e^{d_n}_{1,d_n}(a) \in \Gamma_n$, and $\epsilon_n:=\epsilon_n(1)$.
\item For every $\mathfrak{q} \lhd O$, let $\rho^n_{\mathfrak{q}}:\PSL_n(O)\rightarrow \PSL_n(O/\mathfrak{q})$ be the  quotient map and denote $\PSL_n(O;\mathfrak{q})=\ker \rho_\mathfrak{q}^n$. Let $E_{i,j}^n(A;\mathfrak{q})=E_{i,j}^n(A) \cap \PSL_n(O;\mathfrak{q})$.

\end{enumerate}

\section{Lemmas on $\SL_n(O)$}

\begin{lemma}\label{lemma:ded}
	Let $n \ge 2$ and let $A$ be a Dedekind domain. For every $a \in A^n$ there exists a basis $a_1,\ldots,a_n$ of $A^n$ such that $a \in Aa_1+Aa_2$.
\end{lemma}
\begin{proof} Denote $a=(b_1,\ldots,b_n)$. 
	The result is clear if $a=0$ so we can assume that $b_1 \ne 0$. Denote $\mathfrak{q}=Ab_1 + \ldots+ Ab_n$ and, for every prime ideal $\mathfrak{p}\lhd A$, let $l_\mathfrak{p}$ and $m_{\mathfrak{p}}$ be the maximal natural numbers such that $\mathfrak{p}^{l_{\mathfrak{p}}} \supseteq Ab_1$ and $\mathfrak{p}^{m_{\mathfrak{p}}} \supseteq \mathfrak{q}$ respectively. For all but finitely many $\mathfrak{p}$, $l_{\mathfrak{p}}=m_{\mathfrak{p}}=0$. There exists $B \in \SL_{n-1}(A)$ such that, for every $\mathfrak{p}$ for which $l_{\mathfrak{p}} > m_{\mathfrak{p}}$, the first coordinate of $B(b_2,\ldots,b_n)^t$, denoted by $c$, does not belong to $\mathfrak{p}^{m_{\mathfrak{p}}+1}$. Thus, $\mathfrak{q}=Ab_1+Ac$ and the existence of the required basis is clear. 	
\end{proof}

\begin{lemma}\label{lemma:cong_def} For every $n \geq 3$, a Dedekind domain $A$, a maximal ideal $\mathfrak{q}$ of $A$, and an element $\gamma\in \PSL_n(A) \smallsetminus \PSL_n(A;\mathfrak{q})$, $(\gcl(\gamma)_{\PSL_n(A)})^{32} \cap E_{1,n}^n(A)E_{1,n-1}^n(A)$ is not contained in $E_{1,n}^n(A;\mathfrak{q})E_{1,n-1}^n(A;\mathfrak{q})$.
\end{lemma} 
\begin{proof} Lemma A.3 of \cite{AM} implies that $(\gcl(\gamma)_{\PSL_n(A)})^{32}$ contains a non-identity element which is not equal to the identity modulo $\mathfrak{q}$ and differs from the identity matrix only in the first row.  The result follows from Lemma \ref{lemma:ded}. 
\end{proof}

\begin{definition} \label{def:Ad.relation} For every $n \geq 2$ and ring of $S$-integers $O$, let $\sim$ be the equivalence relation on $\PSL_n(O)$ defined by $g \sim h$ if, for every prime $\mathfrak{p} \lhd O$ and every $a\in O$, 
\[
 (\exists x\in \PSL_n(O)) [x ^{-1} g x,e_{1,n}(1)] \equiv e_{1,n}(a) \text{ (mod $\mathfrak{p}$)} \leftrightarrow 
 \]
 \[
 \leftrightarrow (\exists y\in \PSL_n(O)) [y ^{-1} g h,e_{1,n}(1)] \equiv e_{1,n}(a) \text{ (mod $\mathfrak{p}$)}.
 \]
\end{definition} 

\begin{lemma} \label{lem:Ad.relation} Let $n \geq 2$, let $O$ be a ring of $S$-integers in a number field $K$, and let $\sim$ be the equivalence relation from Definition \ref{def:Ad.relation}. If $g,h\in \PSL_n(O)$ satisfy $g \sim h$, then $\Ad(g),\Ad(h)$ have the same eigenvalues.
\end{lemma} 

\begin{proof} Let $\lambda \neq 1$ be an eigenvalue of $\Ad(g)$ and let $\nu_i$ be the eigenvalues of $\Ad(h)$. There are two eigenvalues $\mu_1,\mu_2$ of $g$ such that $\lambda = \mu_1/\mu_2$. By Chebotarev's density theorem, for infinitely many primes $\mathfrak{p}$, the minimal polynomials of $\mu_i$ split in $O/ \mathfrak{p}$. For those primes, there is a unimodular basis $v_1,\ldots,v_n$ to $O^n$ such that $g v_1=\mu_1v_1 \text{ (mod $\mathfrak{p}$)}$ and $gv_n=\mu_2v_n \text{ (mod $\mathfrak{p}$)}$. Letting $x$ be the matrix representing the change of basis to $v_i$, we get that $[x ^{-1} g x,e_{1,n}(1)]=e_{1,n}(\lambda-1)$. By the assumption, $\lambda=\nu_i \text{ (mod $\mathfrak{p}$)}$, for some $i$. Since this happens for infinitely many primes $\mathfrak{p}$, $\lambda=\nu_i$, for some $i$, so every eigenvalue of $\Ad(g)$ is an eigenvalue of $\Ad(h)$. By symmetry, $\Ad(g),\Ad(h)$ have the same eigenvalues.
\end{proof} 

\begin{definition} \label{def:delta.f} \begin{enumerate} 
\item For a subset $S \subseteq \mathbb{C} ^ \times$, let $\delta S=\left\{ x y ^{-1} \mid x,y\in S \right\}$ and $\delta ^2 S=\delta(\delta(S))$. 
\item Denote
\[
f(n)=\frac{1}{4} n(n-1)(n-2)(n-3)+n(n-1)(n-2)+2n(n-1)+1.
\]
\end{enumerate} 
\end{definition} 

\begin{lemma} \label{lem:delta.square} For every $n \geq 3$, if $S \subseteq \mathbb{C} ^ \times$ with $|S| \leq n$, then $| \delta ^2 S|\leq f(n)$. Equality is obtained for some subset $S \subset \mathbb{C} ^ \times$. If equality holds, then $|S|=n$ and the set $\delta S$ determines $S$ up to multiplication by a scalar.
\end{lemma} 

\begin{proof}If $S=\left\{ x_i \mid 1 \le i \le n\right\}$, then $\delta ^2S= \left\{ \frac{x_i x_j}{x_k x_l}\mid 1 \le i,j,k,l \le n \text{ and } (i,j)\ne(k,l)\right\}$. It follows that $\delta ^2 S$ is the union of the following sets:
\[
A_1=\left\{ \frac{x_i x_j}{x_k x_l} \mid \text{$i,j,k,l$ are all different} \right\} \quad\quad |A_1| \leq \frac{1}{4}n(n-1)(n-2)(n-3),
\]
\[
A_2=\left\{ \frac{x_i^2 }{x_j x_k} \mid \text{$i,j,k$ are all different} \right\} \quad\quad |A_2| \leq \frac{1}{2}n(n-1)(n-2),
\]
\[
A_3=\left\{ \frac{x_i x_j }{x_k^2} \mid \text{$i,j,k$ are all different} \right\} \quad\quad |A_3| \leq \frac{1}{2}n(n-1)(n-2),
\]
\[
A_4=\left\{ \frac{x_i^2 }{x_j^2} \mid i \neq j \right\} \quad\quad |A_4| \leq n(n-1),
\]
\[
A_5=\left\{ \frac{x_i }{x_j} \mid i\neq j \right\} \quad\quad |A_5| \leq n(n-1),
\]
\[
A_6=\left\{ 1 \right\} \quad\quad |A_6|=1,
\]
from which the first claim follows. For the set $S=\left\{ e^{5^i} \mid i=1,\ldots,n \right\}$, the sets $A_i$ are disjoint and achieve the upper bound on their sizes, which implies the second claim. Since $f(n)$ is a strictly increasing function of $n$, the third claim follows. Finally, assume that $| \delta ^2 S|$ achieves the upper bound. Denote $S=\left\{ x_1,\ldots,x_n \right\}$ and define a relation $E$ on $\delta S \smallsetminus \left\{ 1 \right\}$ by $z_1 E z_2$ if $z_1z_2 \in \delta S\smallsetminus \left\{ 1 \right\}$. If $z_1=x_i/x_j$ and $z_2=x_k/x_l$, then $z_1 E z_2$ if and only if either $j=k$ or $i=l$. Picking a pair $z_1,z_2$ such that $z_1 E z_2$, we can assume, without loss of generality, that $z_1=x_a/x_b$ and $z_2=x_c/x_a$. It follows that the set $\left\{ z \in \delta S \mid z E z_1 ^{-1} \wedge \neg (z E z_2 ) \right\} \cup \left\{ 1 \right\}= \left\{ x_i/x_b \mid i \neq b\right\} \cup \left\{ 1 \right\}$ is a scalar multiple of $S$.

\end{proof} 



The following two lemmas are clear:
\begin{lemma} \label{lem:T.vr} Let $n \geq 3$. The set $T^{vr}:=\left\{ t\in (\mathbb{C} ^ \times)^n \mid | \delta ^2 \spec(t) |=f(n) \right\}$ of very regular elements is Zariski open and dense.
\end{lemma} 

\begin{lemma} \label{lem:t12x} Let $t_1,t_2\in T$. If $\spec(t_1x)=\spec(t_2x)$, for all $x$ in a Zariski dense subset, then $t_1=t_2$.
\end{lemma} 


\begin{definition} \begin{enumerate} 
\item $R=\left\{ (t_1,t_2) \in (T^{vr})^2 \mid t_1 \sim t_2 \text{ and } t_1 \neq t_2 \right\}$.
\item We say that $t$ is good if $\dim \left( R \cap (t,t) \cdot R \right) < \dim T$.
\item If $t$ is good, let $T^{t\text{-good}}=T \smallsetminus p_1 \left( R \cap (t,t) \cdot R \right)$, where $p_1: T \times T \rightarrow T$ is the projection on the first coordinate.
\end{enumerate} 
\end{definition} 

\begin{lemma} \label{lem:good.elements} \begin{enumerate} 
\item The set $\left\{ t \in T \mid \text{t is good} \right\}$ is Zariski open and non-empty.
\item If $t$ is good, $t_1\in T^{t\text{-good}}$, and $t_2\in T$ satisfy $t_1 \sim t_2$ and $t_1t \sim t_2t$, then $t_1=t_2$.
\end{enumerate} 
\end{lemma} 

\begin{proof} 

\begin{enumerate} 
\item By Lemma \ref{lem:delta.square}, $\dim R=\dim T$. We claim that, if $X_1$ and $X_2$ are irreducible components of $R$, then the set of $t\in T$ such that $\dim (X_1 \cap (t,t) \cdot X_2)< \dim R$ is Zariski open and dense. Indeed, let $X_1,X_2$ be components of $R$. The set $Y:=\left\{ t \mid \dim (X_1 \cap (t,t)X_2)=\dim R \right\}$ of bad points is constructible. Assume it contains a Zariski open set. Choose a Zariski dense sequence $t_1,t_2,\ldots$ of points of $Y$. By Baire category theorem, there is $x\in X_2$ such that $(t_i,t_i)x\in X_1$, for all $t_i$. If $x=(g,h)$, then this contradicts Lemma \ref{lem:t12x}.

It follows that the set of $t\in T$ such that $\dim(R \cap (t,t)R)<\dim R$ is Zariski open and dense, proving the claim.
\item By definition.
\end{enumerate} 
\end{proof} 



\begin{lemma} \label{lem:torus.span} For every $n \geq 2$ and every maximal torus $T \subseteq \SL_n(\mathbb{C})$, there are elements $\alpha_1,\alpha_2\in \SL_n(\mathbb{C})$ such that $T \cdot (\alpha_1 T \alpha_1^{-1}) \cdot (\alpha_2 T \alpha_2^{-1})$ spans $\Mat_n(\mathbb{C})$.
\end{lemma} 
\begin{proof}
It is enough to prove that claim when $T \subseteq \SL_n(\C)$ consists of diagonal elements. 
\begin{enumerate}
\item\label{item:c2}  $\Span_\C T$ is the set of diagonal matrices in $M_n(A)$. 
\item\label{item:c3} Let $\alpha_1$ be the matrix with 1 on the diagonal entries, 1 on the entries of the right column and zero elsewhere.  Item \ref{item:c2} implies that $\Span_\C (T \cup \alpha_1 T \alpha_1^{-1})$ contains $$\{\beta \in M_n(\C) \mid \text{if }i \ne j \text{ and } j<n\text { then the } (i,j)\text{-entry of }\beta \text{ is equal to }0 \}.$$	
\item\label{item:c4}  Let $\alpha_2$ be the matrix with 1 on the diagonal entries, 1 on the entries of the bottom  row  and zero elsewhere. Item \ref{item:c2} implies that $\Span_\C (T \cup \alpha_2 T \alpha_2^{-1})$ contains $$\{\beta \in M_n(\C) \mid \text{if }i \ne j \text{ and } i<n\text { then the } (i,j)\text{-entry of }\beta \text{ is equal to }0 \}.$$	
\item\label{item:c1}  For every subsets $C,D \subseteq M_n(\C)$,  $$\Span_\C \{cd \mid c \in C\text{ and } d\in D\}=\Span_\C\{cd \mid c \in \Span_\C C\text{ and }d \in \Span_\C D\}.$$
\item Items \ref{item:c3}, \ref{item:c4}, and \ref{item:c1}  imply that $\Span_\C (T \cup \alpha_1 T \alpha_1^{-1})(T \cup \alpha_2 T \alpha_2^{-1})=M_n(\C)$. 
\item $(T \cup \alpha_1 T \alpha_1^{-1})(T \cup \alpha_2 T \alpha_2^{-1})$ is contained in  $T \cdot (\alpha_1 T \alpha_1^{-1}) \cdot (\alpha_2 T \alpha_2^{-1})$.	\end{enumerate}	
\end{proof}

\section{Proof of Theorem \ref{thm:main.bi-int}}

\begin{proof}[Proof of Theorem \ref{thm:main.bi-int}]
Let $$\Upsilon_n:=Z(\Cent_{\Gamma_n}(\epsilon_n))=\left\{ \epsilon_n(a) \mid a\in O_n \right\}=E^{d_n}_{1,d_n}(O_n).$$ By definition, $(\Upsilon_n)_n$ is a definable sequence of subgroups of $(\Gamma_n)_n$. We start by giving an interpretation of $(O_n)_n$ in $(\Gamma_n)_n$.

\begin{claim}\label{claim:O_int} For every $n$, let $\sigma_n,\tau_n\in \Gamma_n$ be such that, for every $a\in O$, $\sigma_n e_{1,d_n}^{d_n}(a) \sigma_n = e_{1,2}^{d_n}(a)$ and $\tau_n e_{1,d_n}^{d_n}(a) \tau_n = e_{2,d_n}^{d_n}(a)$. Let $a_n:\Upsilon_n \times \Upsilon_n \rightarrow \Upsilon_n$ be the multiplication function, let $m_n:\Upsilon_n \times \Upsilon_n \rightarrow \Upsilon_n$ be the function $m_n(x,y)=[\sigma x \sigma, \tau y \tau]$, and let $d_n:\Upsilon_n \rightarrow O_n$ be the function $d_n(\epsilon_n(a))=a$. Denote $\mathcal{D}_n:=(\Upsilon_n, a_n, m_n,d_n)$. The sequence $(\mathcal{D}_n)_n$ is an interpretation of $(O_n)_n$ in $(\Gamma_n)_n$.
\end{claim}

By Theorem \ref{thm:prelim.logic}, the sequence $(O_n)_n$ interprets $(\mathbb{Z})_n$, so there is an interpretation $(\mathcal{F}_n)_n$ of $(\mathbb{Z})_n$ in $(\Gamma_n)_n$. Theorem \ref{thm:prelim.logic} implies that every self interpretation of $(\Z_n)_n$ is trivial. Thus, in order to complete the proof of Theorem \ref{thm:main.bi-int}, we only have to construct an interpretation $(\mathcal{G}_n)_n$ of $(\Gamma_n)_n$ in $(\Z_n)_n$ such that $(\mathcal{G}_n)_n \circ (\mathcal{F}_n)_n$ is trivial.

   \begin{claim}\label{claim:good_int}
   	There exists an interpretation  $(\mathcal{G}_n)_n$ of $(\Gamma_n)_n$ in $(\Z_n)_n$ such that $(h_n^{-1} \restriction_{\Upsilon_n} )_n$ is a sequence of definable  functions where $(\mathcal{H}_n)_n=(H_n,\odot_n,h_n)_n:=(\mathcal{G}_n)_n \circ (\mathcal{F}_n)_n$.
   \end{claim}
\begin{proof}
	Let $\kappa$ be the  Kronecker delta function and denote  $\N^+:=\N \setminus \{0\}$. Let $\mathfrak{M}$ be the set of functions $M:\N^+ \times \N^+ \rightarrow \Z$ such that $M-\kappa$ has a finite support. Godel encoding implies that there exist a bijection $\rho:\Z \rightarrow \mathfrak{M}$ and a definable function $r:\Z \times \N^+ \times \N^+ \rightarrow \Z$ such that for every $z \in \Z$ and every  $i,j \in \N^+$, $r(z,i,j)=\rho(z)(i,j)$.
	
	Theorem \ref{thm:prelim.logic}  implies that there exists an interpretation $((B_n,\oplus_n,\otimes_n,b_n))_n$ of $(O_n)_{n}$ in $(\Z_n)_{n}$ such that, for every $n$, $B_n=\Z$, $b_n: \mathbb{Z} \rightarrow O_n$ is a bijection,  $b_n(0)=0$,  $b_n(1)=1$ and  $(b_n)_n$ is a definable sequence of functions in $(O_n)_n$.

	 For every $M \in \mathfrak{M}$ and every $n$, we denote by $M_n$ the function $M_n:\N^+ \times \N^+ \rightarrow O_n$ defined by $M_n(i,j)=b_n(M(i,j))$. We can view $M_n$ as a matrix over $O_n$  with infinitely many rows and columns  whose $(i,j)$ coordinate is equal to $M_n(i,j)$. There is an obvious definition of $\det M_n$ (i.e the determinant of every large enough finite upper left principle submatrix of $M_n$). 
	 
	 For every $n$, denote
	$$
	\tilde{G}_n=\{\rho^{-1}(M) \mid \left((\forall i,j \in \N^+)\max(i,j)> d_n\rightarrow  M(i,j)-\kappa(i,j)=0\right)\ \wedge\ \det M_n=1\}.
	$$
$(\tilde{G}_n)_n$ is a definable sequence of subsets of $(\Z)_n$. For every $n$, define a function $\tilde{g}_n:\tilde{G}_n \rightarrow \Gamma_n$ in the following way: for every $z \in \tilde{G}_n$, if $M=\rho(z)$ then $\tilde{g}_n(z)$ is the image of $M_n\restriction_{[d_n+1]\times [d_n+1]}$ in $\Gamma_n$. Deonte
$$G_n=:\{z \in \tilde{G}_n \mid (\forall w \in \tilde{G}_n )\ \tilde{g}_n(z)=\tilde{g}_n(w) \rightarrow z \le w \} \text{ and } g_n:=\tilde{g}_n\restriction_{G_n}.$$ 
Then $(G_n)_{n}$ is a definable sequence of subsets of $(\Z)_{n}$ and, for every $n$, $g_n:G_n \rightarrow \Gamma_n$ is a bijection. Finally, define $\cdot_n:G_n \times G_n\rightarrow G_n$ by $z \cdot_n w:=g_n^{-1}(g_n(z)g_n(w))$. We get that $(\mathcal{G}_n)_n=(G_n,\cdot_n,G_n)$ is an interpretation of $(\Gamma)_n$ in $(\Z)_n$. 

Denote $(\mathcal{H}_n)_n=(H_n,\odot_n,h_n)_n:=(\mathcal{G}_n)_n \circ (\mathcal{F}_n)_n$. Then, for every $n$, $H_n=\{\epsilon_n(c) \mid c \in G_n\}$. For every $n$ and every $\epsilon_n(a) \in \Upsilon_n$,  $h_n^{-1}(\epsilon_n(a))=\epsilon_n(c) $ if and only if the following conditions hold: 
\begin{enumerate}
	\item For every $ i \in \N^+$, $r(c,i,i)=1$.
	\item $a=b_n\left(r(c,1,d_n)\right)$. 
	\item If $i \ne j$ and $(i,j)\ne (1,d_n)$ then $r(c,i,j)=0$.
\end{enumerate}
Since the above condtions can be expressed in first order, $(h_n^{-1} \restriction_{\Upsilon_n})_n$ is definable sequence of functions. 
\end{proof}

Fix $(\mathcal{G}_n)_n$ and $(\mathcal{H}_n)_n$ as in Claim \ref{claim:good_int}.
Choose, for any $n$, a regular semisimple element $\theta_n\in \Gamma_n$ such that $\langle \theta_n \rangle$ is Zariski dense in $\Theta_n:=\Cent_{\Gamma_n}(\theta_n)$. Denote $t_n=h_n(\theta_n)$ and $T_n=\Cent_{H_n}(t_n)$. It is clear that $(\theta_n)_n$ is a definable sequence of subsets of $(\Gamma_n)_n$, $(T_n)_n$ is a definable sequence of subsets of $(H_n)_n$, and $h_n(\Theta_n)=T_n$. Our next goal is two show that $(h_n \restriction_{\Theta_n})_n$ is definable. We will prove this in Claim \ref{claim:def.on.Theta} below. Before doing so, we need some preparation.

\begin{claim} \label{prop:cong.Gamma} Let $\mathfrak{C}_n$ be the collection of principal projective congruence subgroups of $\Gamma_n$ corresponding to maximal ideals. Then the sequence $(\mathfrak{C}_n)_n$ is uniformly definable.
\end{claim} 
\begin{proof} 
Since every ideal in a Dedekind domain is generated by two elements, $(\{\mathfrak{q} \lhd O\mid \mathfrak{q}\text{ is maximal}\})_n$ is uniformly definable in $(O_n)_n$. Claim \ref{claim:O_int} implies that  $(\{E^{d_n}_{1,d_n-1}(O;\mathfrak{q})E^{d_n}_{1,d_n}(O;\mathfrak{q}) \mid \mathfrak{q}\lhd A\text{ is maximal} \})_n$ is uniformly definable in $(\Gamma_n)_n$.  Lemma \ref{lemma:cong_def} implies that, for every $n$ and every maximal ideal $\mathfrak{q}\lhd O_n$,  $\alpha \in \PSL_{d_n}(O_n;\mathfrak{q})$ if and only if $\gcl_{\PSL_{d_n}(O_n)}(\alpha)^{32} \cap E^{d_n}_{1,d_n-1}(O_n)E^{d_n}_{1,d_n}(O_n) \subseteq E^{d_n}_{1,d_n-1}(O_n;\mathfrak{q})E^{d_n}_{1,d_n}(O_n;\mathfrak{q}) $. The later condition can be expressed in first order.  \end{proof}

By Claim \ref{prop:cong.Gamma}, there is a definable sequence of sets, $(\Pi_n)_n$ (parameterizing primes) and a definable sequence $(\Sigma_n)_n$ with $\Sigma_n \subseteq \Pi_n \times \Gamma_n$, such that the collection $\left\{ \pr_1 ( \pr_2 ^{-1}(\pi)) \mid \pi \in \Pi_n \right\}=\{\PSL_{d_n}(O_n;\mathfrak{p}) \mid \mathfrak{p} \lhd O_n\text{ is maximal}\}$, where $\pr_i$ are the coordinates projections. Let $(P_n)_n$ be and $(S_n)_n$ be the corresponding sequences to $(\Pi_n)_n$ and $(\Sigma_n)_n$ in $(H_n)_n$. To ease notations, we will write $\Gamma_n[\pi]$ for the congruence subgroup $\pr_1 ( \pr_2 ^{-1}(\pi))$ and $H_n[p]$ for the congruence subgroup $\pr_1( \pr_2 ^{-1} (p))$. We say that $\pi \in \Pi_n$ and $p\in P_n$ are compatible if $h_n^{-1}(\Gamma_n[\pi])=H_n[p]$ (informally, they define the same congruence subgroup).

\begin{claim} \label{lem:primes.compatible} The sequence $\left( \left\{ (\pi,p)\in \Pi_n \times P_n \mid \text{$\pi$ and $p$ are compatible}  \right\} \right)_n$ is definable.
\end{claim} 

\begin{proof} The condition that $\pi$ and $p$ are compatible is equivalent to $h_n(\pr_1(\pr_2 ^{-1} (\pi)) \cap \Upsilon_n)=\pr_1(\pr_2 ^{-1} (p)) \cap U_n$. The claim now follows because $(h_n^{-1} \restriction_{\Upsilon_n})_n$ is definable.
\end{proof}


\begin{claim} \label{claim:def.on.Theta} The sequence of functions $(h_n^{-1} \restriction_{\Theta_n})_n$ is definable. 
\end{claim} 

\begin{proof} Let $\delta$ and $f$ be as in Definition \ref{def:delta.f}. We first claim that the sequence $(\Theta_n^{vr})_n$ is definable. Indeed, for $\theta \in \Theta_n$ and a prime $\mathfrak{p} \triangleleft O_n$, let
\[
\Sigma_{\theta,\mathfrak{p}}=\left\{ a \in O_n/\mathfrak{p} \mid (\exists x \in \Gamma_n) [x ^{-1} \theta x,\epsilon_n(1)]=\epsilon_n(a-1) \text{ (mod $\mathfrak{p}$)} \right\},
\]
and let $\Delta_{\theta,\mathfrak{p}}=\delta \Sigma_{\theta,\mathfrak{p}}$. For any $\theta$ and $\mathfrak{p}$, the set $\Sigma_{\theta,\mathfrak{p}}$ is a subset of the set of eigenvalues of $\Ad(\theta)$ modulo $\mathfrak{p}$, and $| \Delta_{\theta,\mathfrak{p}}| \leq |\delta ^2 \spec(\theta)|$. By Chebotarev's density theorem, for every $\theta$, there are infinitely many primes $\mathfrak{p}$ such that $\Sigma_{\theta,\mathfrak{p}}$ is a equal to the set of eigenvalues of $\Ad(\theta)$ modulo $\mathfrak{p}$ and $| \Delta_{\theta,\mathfrak{p}} | = |\delta ^2 \spec(\theta)|$. Since $(\mathbb{Z})_n$ is bi-interpretable with $(O_n)_n$ and since
\[
\Theta_n^{vr}=\left\{ \theta \in \Theta_n \mid \text{ there exists  a prime $\mathfrak{p}$ such that  $| \Delta_{\theta,\mathfrak{p}} |=f(d_n)$} \right\}, 
\]
we get that $(\Theta_n^{vr})_n$ is definable.


For every $n$, let $\sigma_n \in \Theta_n$ be a good element provided by Lemma \ref{lem:good.elements}, and let $s_n=h_n(\sigma_n)\in T_n$. By definition, the sequences $\left( \Theta_n^{\sigma_n\text{-good}} \right)_n$ and $\left( T_n^{s_n\text{-good}} \right)_n$ are definable. Let $A_n$ be the set of pairs $(\gamma,g)\in \Theta_n^{\sigma_n\text{-good}} \times T_n^{s_n\text{-good}}$ such that, for any compatible $(\pi,p)\in \Pi_n \times P_n$ and every $\upsilon\in \Upsilon_n$,
\[
 (\exists x\in \Gamma_n) [x ^{-1} \gamma x,\epsilon_n] \equiv \upsilon \text{ (mod $\Gamma_n[\pi]$)} \leftrightarrow (\exists y\in H_n) [y ^{-1} g y,e_{n}] \equiv h(\upsilon) \text{ (mod $H_n[p]$)}
\]
and
\[
 (\exists x\in \Gamma_n) [x ^{-1} \gamma \sigma_n x,\epsilon_n] \equiv \upsilon \text{ (mod $\Gamma_n[\pi]$)} \leftrightarrow (\exists y\in H_n) [y ^{-1} g s_n y,e_{n}] \equiv h(\upsilon) \text{ (mod $H_n[p]$)}.
\]
By Lemma \ref{lem:primes.compatible}, the sequence $(A_n)_n$ is definable and, by Lemma \ref{lem:good.elements}, $A_n$ is the graph of the restriction of $h_n$ to $\Theta_n^{\sigma_n\text{-good}}$.

Finally, since $\Theta_n^{\sigma_n\text{-good}}$ are Zariski open and dense, $\left( \Theta_n^{\sigma_n\text{-good}} \right) ^2=\Theta_n$, and the sequence $\left( h_n\restriction_{\Theta_n} \right)_n$ is definable.
\end{proof} 

\begin{claim} The sequence $(h_n)_n$ is definable.
\end{claim} 

\begin{proof} For every $n$, let $R_n \subset \Gamma_n \times H_n$ be the collection of pairs $(\gamma,g)$ such that, for every compatible pair $(\pi,p)$, the following holds:
\[
\left( \left( \forall \xi \right)\left( \xi ^{-1} \gamma \xi \in \Theta_n \Gamma_n[\pi] \right)  \rightarrow \left( \left( \exists x \right) h_n( \xi ^{-1} \gamma \xi \Gamma[\pi])= x^{-1} g x H[p] \right) \right).
\]
By Claim \ref{claim:def.on.Theta}, the sequence $(R_n)_n$ is definable. Clearly, the graph of $h_n$ is contained in $R_n$. If $(\gamma,g)\in R_n$, then, by Chebotarev's density theorem, 
$\trace(h_n(\gamma))=\trace(g)$.

Applying Lemma \ref{lem:torus.span} to $\Gamma_n$ and $\Theta_n$, we get elements $\alpha_n,\beta_n\in \Gamma_n$ such that  the lift of $\Theta_n \alpha_n^{-1} \Theta_n \alpha_n \beta_n ^{-1} \Theta_n \beta_n$ to $\SL_{d_n}(O_n)$ spans $M_{d_n}(\C)$. Denote $h_n^{-1}(\alpha_n)= a_n$ and $h_n ^{-1}(\beta_n) =b_n$. By Claim \ref{claim:def.on.Theta}, the sequence of functions $(h_n\restriction_{\Theta_n \alpha_n ^{-1} \Theta_n \alpha_n \beta_n^{-1} \Theta_n \beta_n})_n$ is definable. Since  the lift of $\Theta_n \alpha_n^{-1} \Theta_n \alpha_n \beta_n ^{-1} \Theta_n \beta_n$ to $\SL_{d_n}(O_n)$ spans $M_{d_n}(\C)$ and, we have that $h_n(\gamma)=g$ if and only if
\[
\left( \forall \lambda \in \Theta_n \alpha_n ^{-1} \Theta_n \alpha_n \beta_n^{-1} \Theta_n \beta_n \right) (\gamma \lambda, g h_n(\lambda))\in R_n.
\]
Since the last sequence of relations is definable, the claim follows.
\end{proof} 
The proof of Theorem \ref{thm:main.bi-int} is now complete. 
\end{proof}

\end{document}